\documentclass[12pt,leqno,a4paper]{amsart}
\usepackage{graphicx}
\usepackage{amsmath,amssymb,amsthm}
\usepackage{enumitem}
\usepackage[margin=3cm,top=4cm,bottom=3.5cm,footskip=1cm,headsep=1.5cm]{geometry}
\usepackage{amsaddr}
\usepackage[utf8]{inputenc}
\usepackage{color}

\def\blue{\color{blue}}

\author{H. Egger$^*$ \and B. Hofmann$^\dag$}
\address{$^*$Department of Mathematics, TU Darmstadt, Germany\\%
$^\dag$Faculty of Mathematics, TU Chemnitz, Germany}
\email{egger@mathematik.tu-darmstadt.de}
\email{hofmannb@mathematik.tu-chemnitz.de}

\title[Tikhonov regularization under conditional stability]{Tikhonov regularization in Hilbert scales\\ under conditional stability assumptions}

\newtheorem{lemma}{Lemma}[section]

\newtheorem{theorem}[lemma]{Theorem}

\theoremstyle{definition}
\newtheorem{remark}[lemma]{Remark}

\def\RR{\mathbb{R}}
\def\ZZ{\mathbb{Z}}

\def\D{\mathcal{D}}
\def\X{\mathcal{X}}
\def\Y{\mathcal{Y}}

\def\xdag{x^\dag}
\def\xad{x_\alpha^\delta}
\def\xand{x_{\alpha_n}^\delta\!}
\def\xansd{x_{\alpha_{n^*}}^\delta\!}
\def\yd{y^\delta}

\def\cad{c_\alpha^\delta}
\def\cdag{c^\dag}
\def\fad{f_\alpha^\delta}
\def\fdag{f^\dag}

\def\M{\mathcal{M}}

\numberwithin{equation}{section}
\begin{document}

%{\em Draft: \today}

\begin{abstract}
Conditional stability estimates allow us to characterize the degree of
ill-posedness of many inverse problems, but without further
assumptions they are not sufficient for the stable solution in the
presence of data perturbations. We here consider the stable solution
of nonlinear inverse problems satisfying a conditional stability
estimate by Tikhonov regularization in Hilbert scales. Order optimal
convergence rates are established for a-priori and a-posteriori
parameter choice strategies. The role of a hidden source condition is
investigated and the relation to previous results for regularization
in Hilbert scales is elaborated. The applicability of the results is
discussed for some model problems, and the theoretical results are
illustrated by numerical tests.
\end{abstract}

\maketitle

\begin{quote}
\end{quote}

\begin{quote}
\noindent
{\small {\bf AMS-classification (2010):}
47J06, 49N45, 65J22, 35R30
}
\end{quote}

\section{Introduction} \label{sec:intro}

We consider nonlinear operator equations of the abstract form
\begin{align}  \label{eq:ip}
F(x)= y\,,
\end{align}
modeling inverse problems with a nonlinear forward operator $F:\D(F) \subset \X \to \Y$  mapping between Hilbert spaces $\X$ and $\Y$. We assume that $y=F(\xdag)$ for some $\xdag \in \D(F)$, i.e., the exact data $y$ result from
an element $\xdag$ in the domain of $F$, which we call the true solution of \eqref{eq:ip}.
In practice, only noisy data $\yd \in \Y$ are available and we try to approximate the solution $\xdag$ from knowledge of $\yd$ under the assumption that a deterministic bound for the noise level is available, i.e.
\begin{align} \label{eq:a1}
\|\yd-y\|_\Y \le \delta. \tag{A1}
\end{align}
As a consequence of the \emph{smoothing} properties of $F$, which are typical for inverse problems, we must expect that the
equation \eqref{eq:ip} is \emph{locally ill-posed} at $\xdag$ in the sense of \cite[Def.~3]{HofPla18}, and some sort of regularization is required in order to determine a stable approximation for $\xdag$.
In this paper, we consider Tikhonov regularization in Hilbert scales and analyze its convergence under a conditional stability assumption for the inverse problem.

In order to state this assumption and the resulting regularization scheme, we use a Hilbert scale $\{\X_s\}_{s \in \RR}$ generated by a densely defined injective self-adjoint linear operator $L : \D(L) \subset \X \to \X$ with compact inverse. The elements of $\X_s$ have finite norm $\|x\|_{\X_s} := \|L^s x\|_{\X}$.

The basic assumption for the rest of the manuscript is that the operator $F$ satisfies a conditional stability estimate of the following form:
There exists a function \linebreak $R : \RR_+ \to \RR_+$ and real numbers $a \ge 0$ and $s,\gamma$ with $-a \le s$ and $0<\gamma \le 1$ such that
\begin{align} \label{eq:a2}
\|x_1-x_2\|_{\X_{-a}} \le R(\rho) \|F(x_1)-F(x_2)\|_{\Y}^\gamma
\qquad \text{for all } x_1,x_2 \in \M_\rho^s\,,\tag{A2}
\end{align}
where the set $\M_\rho^s$ is defined by
$\;\M_\rho^s:=\{x \in \D(F) \cap \X_s:\, \|x\|_{\X_s} \le \rho\}$.

No additional properties of $F$ and $\D(F)$ apart from assumption (A2) will be required for our analysis.
In particular, $F$ may not be differentiable or even discontinuous.

For the stable solution of the inverse problem \eqref{eq:ip} in the presence of data noise,
we utilize Tikhonov regularization in Hilbert scales which is based on minimization of the regularized least-squares functional
\begin{align} \label{eq:tik}
T(x;\alpha,\yd) := \|F(x) - \yd\|^2_\Y + \alpha \|x\|_{\X_s}^2.
\end{align}
As regularized approximations for the solution of \eqref{eq:ip}, we consider approximate minimizers $\xad$ of the functional \eqref{eq:tik}, i.e., arbitrary elements from the set
\begin{align} \label{eq:Xad}
\X(\alpha;\yd) :=\{x \in \D(F) : T(x;\alpha,\yd) \le \inf_{z \in \D(F)} T(z;\alpha,\yd) + \delta^2\}.
\end{align}
The following convergence rate result can then be deduced in a similar way than the more general statement of \cite[Thm.~2.1]{ChengYamamoto00}. We present here a reformulation adopted to our setting and provide a short proof for later reference.
\begin{theorem}{\em \bf (compare with \cite[Thm.~2.1]{ChengYamamoto00} and \cite[Prop.~6.9]{HofmannYamamoto10}).}  \label{thm:previous} $ $\\
Let \eqref{eq:a1}--\eqref{eq:a2} hold and assume that $\xdag\in \D(F) \cap \X_s$.
Then for any $\xad \in \X(\alpha;\yd)$ with $\alpha=\delta^2$ there holds
\begin{align*}
\|F(\xad)-\yd\|_\Y \le C_1 \delta
\qquad \text{and} \qquad
\|\xad - \xdag\|_{\X_{-a}}  \le C_2 \delta^{\gamma}
\end{align*}
with constants $C_1,C_2$ depending only on the size of $\|\xdag\|_{\X_s}$.
Moreover, $\|\xad\|_{\X_{s}} \le C_1$.
\end{theorem}
\begin{proof}
Set $M=\|\xdag\|_{\X_s}$.
Then from the definition of $\xad$ and the choice $\alpha=\delta^2$ for the regularization parameter,
one can deduce that
\begin{align*}
\|F(\xad)-\yd\|^2_\Y + \alpha \|\xad\|^2_{\X_s}
&\le \|F(\xdag) - \yd\|^2_\Y + \alpha \|\xdag\|_{\X_s}^2 + \delta^2\\
&\le 2 \delta^2 + \alpha M^2 = C_1^2 \delta^2
\end{align*}
with constant $C_1^2 := 2 + M^2$.
This already yields the first estimate of the theorem and also provides the following bound for the minimizers
\begin{align*}
\|\xad\|_{\X_s}^2 \le C_1^2 \delta^2/\alpha = C_1^2.
\end{align*}
Hence $\xad,\xdag \in \D(F) \cap \X_s$ with $\|\xad\|_{\X_s},\|\xdag\|_{\X_s} \le C_1$.
From assumption \eqref{eq:a2}, we can then infer that
\begin{align*}
\|\xad-\xdag\|_{\X_{-a}}
&\le R(C_1) \|F(\xad) - F(\xdag)\|_\Y^\gamma \\
&\le R(C_1) (\|F(\xad)-\yd\|_\Y + \|y-\yd\|_\Y)^\gamma
 \le R(C_1) (C_1+1)^\gamma \delta^\gamma.
\end{align*}
This yields the second estimate with constant $C_2=R(C_1) (C_1+1)^\gamma$.
\end{proof}

The proof of the theorem is rather simple and can be extended to Banach spaces and more general conditional stability assumptions and distance measures; let us refer to \cite[Thm.~2.1]{ChengYamamoto00} and \cite[Thm.~1]{ChengHofmannLu14} for details in this direction.
The assertions of the theorem are however quite remarkable, namely:
\begin{itemize}\itemsep0ex
 \item The assumptions on the forward operator $F$ are very general. In particular, no further conditions concerning the {\sl continuity} or {\sl nonlinearity} of $F$ apart from \eqref{eq:a2} are required.
       It should be mentioned, however, that \eqref{eq:a2} implies the injectivity of $F$ with respect to the subset $\M_\rho^s$ of $\X$ and consequently uniqueness of the solution $\xdag$ to equation \eqref{eq:ip} in $\M_\rho^s$. The number $a \ge 0$ can be considered as the \emph{degree of ill-posedness} for the inverse problem under consideration.
 \item Only the simple {\sl source condition} $\xdag \in \D(F) \cap \X_s$ is used,
       but no a-priori bound for $\|\xdag\|_{\X_s}$ is required in the formulation of the method and its analysis.
       Note that this source condition together with the stability condition \eqref{eq:a2}, but without knowledge of a bound for $\|\xdag\|_{\X_s}$, is not sufficient to ensure convergence of approximate solutions without additional regularization.
       %
%        As can be seen from the proof of the theorem, Tikhonov regularization with a sufficiently large regularization parameter guarantees that the regularized solutions $x_\alpha^\delta$ belong to the set $\M_\rho^s$ for some $\rho \approx \|\xdag\|_{\X_s}$ which allows to take advantage of assumption (A2).
 \item The a-priori parameter choice $\alpha=\delta^2$ already yields $\|\xad-\xdag\|_{\X_{-a}} = O(\delta^\gamma)$ which can be seen to be {\sl order optimal} under the given assumptions.
\end{itemize}
% \end{remark}
%

Without much difficulty, one can also establish convergence rates in intermediate norms.
Let us recall the well-known interpolation inequality in Hilbert scales
\begin{align} \label{eq:intpol}
\|x\|_{\X_{q}} \le \|x\|_{\X_{p}}^{\frac{r-q}{r-p}} \|x\|_{\X_r}^{\frac{q-p}{r-p}}, \qquad p \le q \le r.
\end{align}
From the results of Theorem~\ref{thm:previous}, we can then deduce the following estimates.
\begin{theorem} \label{thm:rate}
Under the assumptoins of Theorem~\ref{thm:previous}, one has
\begin{align} \label{eq:rate}
\|\xad-\xdag\|_{\X_r} \le C_3 \delta^{\gamma \frac{s-r}{s+a}} \quad \text{for all } -a \le r \le s ,
\end{align}
with constant $C_3$ depending again only on the size of $\|\xdag\|_{\X_s}$.
\end{theorem}
\begin{proof}
By Theorem~\ref{thm:previous}, we know that $\|\xdag\|_{\X_s},\|\xad\|_{\X_s} \le C_1$ and consequently
\begin{align*}
\|\xad-\xdag\|_{\X_s} \le 2C_1 \qquad \text{and} \qquad \|\xad-\xdag\|_{\X_{-a}} \le C_2 \delta^{\gamma}.
\end{align*}
The assertion then follows immediately from the interpolation estimate \eqref{eq:intpol}.
\end{proof}
Theorems~\ref{thm:previous} and \ref{thm:rate} can be seen as an extension of the convergence results for \emph{regularization in Hilbert scales} that were obtained in \cite{Neubauer92} and \cite{Tautenhahn94} under stronger assumptions on the operator $F$;
see also \cite{Natterer84} or \cite[Section~8.4]{EnglHankeNeubauer96} for linear problems.
In view of these results, one might hope to obtain improved rates
\begin{align} \label{eq:rate2}
\|\xad - \xdag\|_{\X_{r}} \le C_4 \delta^{\gamma \frac{u-r}{u+a}} \quad \text{for} \quad -a \le r \le u,
\end{align}
if the solution satisfies a stronger source condition $\xdag \in \D(F) \cap \X_u$ for some $u \ge s$.

This is indeed not unrealistic:
If the smoothness index $u$ of the solution would be known, one might choose $s=u$ in the regularization term
of the Tikhonov functional \eqref{eq:tik} and also in all previous results.
The improved convergence rate estimate \eqref{eq:rate2} would then follow directly from Theorem~\ref{thm:rate}.
In practice one can, however, not assume knowledge of the smoothness index $u$,
and we therefore stay with the Tikhonov functional defined in \eqref{eq:tik}.
In this case the simple parameter choice $\alpha=\delta^2$ however yields only the suboptimal rates \eqref{eq:rate}
instead of \eqref{eq:rate2} in general.
The main focus of our paper will therefore be to
\begin{itemize}\itemsep0ex
 \item devise an appropriate a-priori parameter choice rule and prove the optimal convergence rates for this choice when $\xdag \in \X_u$;
 \item propose an a-posteriori parameter choice strategy that yields the optimal rates without knowing the smoothness index $u$ of the true solution;
\item discuss the relation to previous results for regularization in Hilbert scales.
\end{itemize}
For illustration of our results, the estimates of Theorems~\ref{thm:previous} and \ref{thm:rate}
and the new estimates obtained in the paper are depicted
in Figure~\ref{fig:rates}.
\begin{figure}[ht!]
\hspace*{2em}\includegraphics[width=0.6\textwidth]{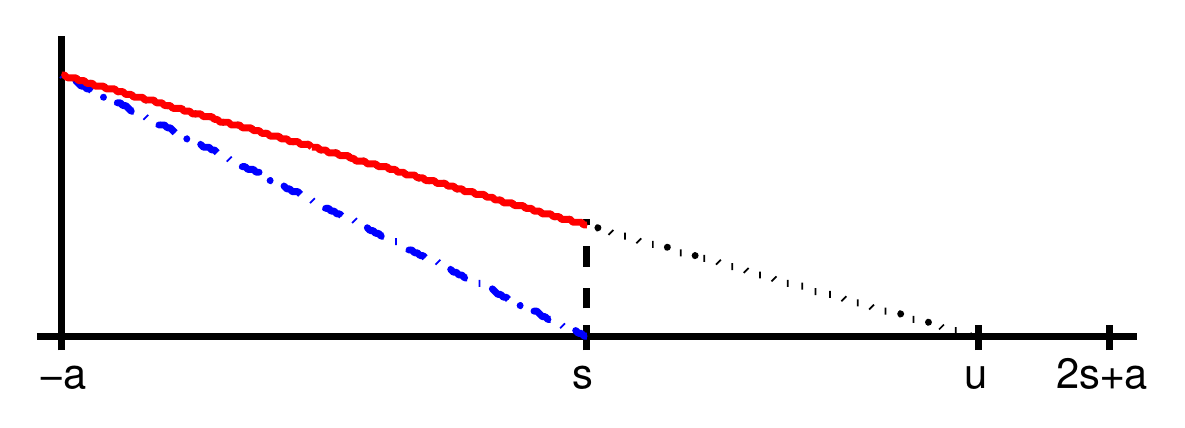}
\put(-313,42){\color{blue} \small \bf Theorems \, \ref{thm:previous} and \ref{thm:rate}}
\put(-170,62){\color{red} \small \bf Theorems~\ref{thm:improve} and \ref{thm:improve2}}
% \put(-80,35){\color{black} \small \bf Theorem~\ref{thm:improve3}}
\put(-3,17){\color{black} \small $\bf r$}
\put(-253,73){\color{black} \small $\kappa$}
\put(-245,76){\thicklines \line(1,0){5}}
\caption{Exponents $0 \le \kappa \le 1$ of $\|\xad-\xdag\|_{\X_r}=O(\delta^\kappa)$ as $\delta~\to~0$ under the assumption that $\xdag \in \D(F) \cap \X_u$.
Estimate \eqref{eq:rate} of~Theorems~\ref{thm:previous} and \ref{thm:rate} (blue, dashed);
estimate \eqref{eq:rate3} obtained by a-priori and a-posteriori parameter choice in Theorems~\ref{thm:improve} and \ref{thm:improve2} (red).
% estimate \eqref{eq:rate4} provided by Theorem~\ref{thm:improved3} under stronger assumptions (black, dotted).
\label{fig:rates}}
\end{figure}
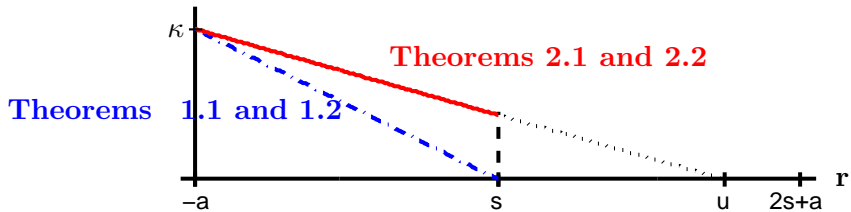
Note that we obtain better convergence rates than predicted by Theorem~\ref{thm:previous}, whenever $\xdag \in \X_u$ with $u > s$.
% For the a-priori parameter choice strategy, the optimal convergence rates can be guaranteed also for stronger norms than the one used in the Tikhonov regularization.
Let us also mention that similar results were obtained by Tautenhahn \cite{Tautenhahn94} under stronger assumptions on the problem, i.e., under a Lipschitz stability condition for the inverse problem, further nonlinearity conditions on the operator $F$, and using an exact minimization of the Tikhonov functional in the definition of $\xad$.

The remainder of the manuscript is organized as follows:
In Sections~\ref{sec:main} and \ref{sec:proofs}, we state and prove our main results.
The relation to previous results on regularization in Hilbert scales will be discussed in Section~\ref{sec:relation}.
Section~\ref{sec:problems} presents two model problems, for which our
results apply, and in Section~\ref{sec:numerics} we illustrate our theoretical results by numerical tests.
The presentation closes with a short discussion of open topics.

\section{Statement of the main results} \label{sec:main}

We now state our main results, which can be seen as a generalization and improvement of the estimates of Theorem~\ref{thm:previous} for smooth solutions $\xdag \in \X_u$ with $u > s$.
\begin{theorem}[A-priori parameter choice] \label{thm:improve} $ $\\
Let \eqref{eq:a1}--\eqref{eq:a2} hold and assume that $\xdag \in \D(F) \cap \X_u$ for some $-a \le s \le u \le 2s+a$.
Then for any $\xad \in \X(\alpha;\yd)$ with $\alpha = \delta^{2-2 \gamma \frac{u-s}{u+a}}$ one has
\begin{align} \label{eq:rate3}
\|\xad - \xdag\|_{\X_r} \le C_5 \delta^{\gamma \frac{u-r}{u+a}}
\qquad \text{for all } -a \le r \le s
\end{align}
and $\|F(\xad)-\yd\|_\Y \le C_6 \delta$ with $C_5,C_6$ depending only on the norm $\|\xdag\|_{\X_u}$
\end{theorem}
For the limiting case $u=s$, we recover the assertions of Theorem~\ref{thm:previous}.
Also note that Theorem~\ref{thm:improve} is yet not fully practical, since the a-priori parameter choice still requires knowledge
of the smoothness index $u$ of the true solution which is not known in practice.
An a-posteriori parameter choice strategy that does not require such knowledge is the discrepancy principle \cite{Morozov84}. Here we consider the following strategy adopted to our setting:
For $n \ge 0$ set $\alpha_n = 2^{-n}$ and choose $\xand \in X(\alpha_n;\yd)$.
Based on this sequence $\{\xand\}_{n \ge 0}$, we define
\begin{align} \label{eq:stopping}
n^* := \inf\{n\ge 0: \|F(\xand)-\yd\|_\Y \le 4 \delta\},
\end{align}
and we consider $\alpha=\alpha_{n^*}$ as the regularization parameter for determing a stable approximation for the solution
of the inverse problem \eqref{eq:ip}.
This parameter choice strategy is computationally attractive since one can start generation
of $\xand$ with large $\alpha_n$, proceed to smaller $\alpha_n$, and stop as soon as the data residual
reaches the desired tolerance.
With this simple procedure, we already obtain the following result.
\begin{theorem}[A-posteriori parameter choice] \label{thm:improve2} $ $\\
Let \eqref{eq:a1}--\eqref{eq:a2} hold, define $n^*$, $\alpha_{n^*}$ as above,
and assume that $\xdag \in \D(F) \cap \X_u$ for some $-a \le s \le u \le 2s+a$.
Then for any $\xad \in \X(\alpha_{n^*};\yd)$ the rates \eqref{eq:rate3} hold
with a constant $C_5$ depending only on $\|\xdag\|_{\X_u}$.
Moreover, we have $\|F(\xad)-\yd\|_\Y \le 4\delta$.
\end{theorem}
Note that no information about the actual smoothness of $\xdag$ is required in order to carry out the computations.
Theorems~\ref{thm:improve} and \ref{thm:improve2} therefore yield a generalization and
a true improvement of Theorem~\ref{thm:previous} for smooth solutions $\xdag \in \X_u$ with $u>s$.

% \bigskip
%
% Under additional assumptions on the operator $F$ it is possible to obtain also convergence in stronger norms
% than used for the regularization term. For the following result, we assume that $F$ is continuously differentiable and that
% \begin{align} \label{eq:a3}
% \|F'(x)\|_{\X_{-a} \to \Y} \le C', \qquad x \in \X_s.  \tag{A3}
% \end{align}
% Under this rather strong nonlinearity condition, we can show the following result.
% \begin{theorem} \label{thm:improve3}
% Let \eqref{eq:a1}--\eqref{eq:a3} hold with $\D(F)=\X_s$ and assume that $\xdag \in \X_u$ for some $s \le u \le 2s+a$.
% Moreover, let $\xad$ be a minimizer of the Tikhonov functional \eqref{eq:tik} with regularization parameter $\alpha$ chosen as in Theorem~\ref{thm:improve}
% or $\alpha=\alpha_{n^*}$ with $n^*$ determined as in \eqref{eq:stopping}.
% Then the estimates \eqref{eq:rate3} hold for all $-a \le r \le 2s+a$.
% \end{theorem}
% \noindent
% Note that we also need to minimize the Tikhonov functional exactly for this result.

\section{Proof of the main results} \label{sec:proofs}

This section is devoted to the proofs of Theorems~\ref{thm:improve} and \ref{thm:improve2}.
Throughout this section, we will always assume that assumptions \eqref{eq:a1}--\eqref{eq:a2} are valid
and we assume that $-a \le s \le u \le 2s+a$.
Let us start with some preliminary observations and auxiliary a-priori estimates.
From the definition of the set $\X(\alpha;\yd)$ and the estimates in the proof of Theorem~\ref{thm:previous},
we directly obtain the following results.
\begin{lemma} \label{lem:nonempty}
Let $\xdag \in \D(F) \cap \X_s$. Then $\X(\alpha;\yd)$ is non-empty for any $\alpha \ge 0$.
Moreover, we have $\xdag \in \X(\alpha;\yd)$ if $\alpha$ is chosen sufficiently small.
\end{lemma}
From the estimates derived in the proof of Theorem~\ref{thm:previous}, we also obtain the following
a-priori estimates for the data residuals and the approximate minimizers.
\begin{lemma} \label{lem:bounds}
Let $\xdag \in \D(F) \cap \X_s$ with $\|\xdag\|_{\X_s}=M$ and let $\tau>0$ be arbitrary.
Then
\begin{itemize}\itemsep1ex
\item[(i)] $\|F(\xad) - \yd\|_\Y^2 \le (2+\tau) \delta^2$ for all $\alpha \le \tau \delta^2/M^2$;
\item[(ii)] $\|\xad\|_{\X_s}^2 \le (1+\frac{2}{\tau}) M^2$ when $\alpha \ge \tau \delta^2/M^2$.
\end{itemize}
\end{lemma}
The parameter $\tau$ was introduced here only to have some flexibility that will be required in our proofs further below.
In order to establish the improved convergence rates \eqref{eq:rate3}, we will need the following more accurate a-priori estimates.
\begin{lemma} \label{lem:rate}
Let $\xdag \in \D(F) \cap \X_u$ with $s < u \le 2s+a$ and let $\xad \in \X(\alpha;\yd)$ for some parameter $\alpha \ge C \delta^2$ with $C>0$ only depending on $\|\xdag\|_{\X_u}$. Then
\begin{align} \label{eq:apriori}
\|F(\xad) - F(\xdag)\|_\Y^2 + \alpha \|\xad - \xdag\|_{\X_s}^2 \le 12 \delta^2 + C_7 \alpha^{\frac{u+a}{u+a-\gamma(u-s)}}
\end{align}
with constant $C_7$ only depending on the size  of $\|\xdag\|_{\X_u}$.
\end{lemma}
\begin{proof}
From the lower bound bound on $\alpha$ stated in the assumptions, the definition of $\xad$, and the second estimate of Lemma~\ref{lem:bounds}, one can see that $\xad \in \D(F) \cap \X_s$ with $\|\xad\|_{\X_s} \le M_1$ for some constant $M_1$ only depending on the norm $\|\xdag\|_{\X_u}$.
With similar arguments as in the proof of \cite[Thm.~10.4]{EnglHankeNeubauer96},
we then get
\begin{align} \label{eq:est}
\frac{1}{2} \|F(\xad) &- F(\xdag)\|_\Y^2 + \alpha \|\xad - \xdag\|^2_{\X_s} \\
&\le  \|F(\xad) - \yd\|_\Y^2 + \delta^2 + \alpha \|\xad - \xdag\|^2_{\X_s} \notag\\
&\le \|F(\xdag)-\yd\|^2_\Y + 2 \delta^2 + 2\alpha (\xdag,\xdag-\xad)_{\X_s} \notag\\
&\le  3 \delta^2 + 2 \alpha \|\xdag\|_{\X_u} \|\xad - \xdag\|_{\X_{2s-u}}. \notag
\end{align}
Here we used \eqref{eq:a1} in the first inequality,
the definition of $\xad$ and some basic computations for the second estimate,
and assumption \eqref{eq:a1} as well as a Cauchy-Schwarz-type inequality in the third step.
Application of the interpolation inequality \eqref{eq:intpol} with $p=-a$, $q=2s-u$, and $r=s$,
further noting that $\|\xad\|_{\X_s}\le M_1$ and $\|\xdag\|_{\X_s} \le \|\xdag\|_{\X_u} \le M_1$ w.l.o.g.,
and using the stability estimate \eqref{eq:a2} yields
\begin{align*}
\|\xad - \xdag\|_{\X_{2s-u}}
&\le \|\xad-\xdag\|_{\X_{-a}}^{\frac{u-s}{s+a}} \|\xad - \xdag\|_{\X_{s}}^{\frac{2s+a-u}{s+a}} \\
&\le R(M_1)^{\frac{u-s}{s+a}}  \|F(\xad)-F(\xdag)\|_\Y^{\gamma\frac{u-s}{s+a}} \|\xad - \xdag\|_{\X_{s}}^{\frac{2s+a-u}{s+a}} .
\end{align*}
This allows us to bound the last term in the estimate \eqref{eq:est} by
\begin{align} \label{eq:est2}
&2 \alpha \|\xdag\|_{\X_u} \|\xad - \xdag\|_{\X_{2s-u}}\\
&\le \big(C' \alpha^\frac{u+a}{2s+2a} \|\xdag\|_{\X_u}\big) \cdot \big(\alpha^\frac{2s+a-u}{2s+2a} \|\xad-\xdag\|_{\X_s}^{\frac{2s+a-u}{s+a}}\big) \cdot \big(\|F(\xad)-F(\xdag)\|^{\gamma \frac{u-s}{s+a}}\big), \notag
\end{align}
with constant $C'$ only depending on $\|\xdag\|_{\X_u}$.
By Young's inequality, one can see that
\begin{align*}
a \cdot b \cdot c \le \frac{a^p}{p} + \frac{b^q}{q} + \frac{c^r}{r}
\end{align*}
for all $a,b,c \ge 0$ and $1 \le p,q,r \le \infty$ with $ \frac{1}{p}+\frac{1}{q}+\frac{1}{r} =1$.
Applying this inequality with indices
$p=\frac{2s+2a}{u+a-\gamma(u-s)}$, $q=\frac{2s+2a}{2s+a-u}$, and $r=\frac{2}{\gamma}\frac{s+a}{u-s}$ in the estimate \eqref{eq:est2} leads to
\begin{align*}
2 \alpha \|\xdag\|_{\X_u} &\|\xad - \xdag\|_{\X_{2s-u}}\\
&\le C'' \alpha^\frac{u+a}{u+a-\gamma(u-s)} + \tfrac{3\alpha}{4} \|\xad - \xdag\|_{\X_s}^2 + \tfrac{1}{4} \|F(\xad)-F(\xdag)\|^2_\Y,
\end{align*}
with $C''$ depending on $\|\xdag\|_{\X_u}$ and on the indices $a,s,u$, and $\gamma$.
The last two terms can then be absorbed in the left hand side of \eqref{eq:est}.
This yields the estimate of the lemma with constant $C_7=4 C''$ and concludes the proof.
\end{proof}

\paragraph{\bf Proof of Theorem~\ref{thm:improve}.}
A combination of Lemma~\ref{lem:rate} with the a-priori parameter choice $\alpha = \delta^{2-2 \gamma \frac{u-s}{u+a}}$ yields $\|F(\xad)-F(\xdag)\|_\Y \le C \delta$ and $\|\xad-\xdag\|_{\X_s} \le C' \delta^{\gamma \frac{u-s}{u+a}}$. Using assumption \eqref{eq:a2}, we also obtain $\|\xad-\xdag\|_{\X_{-a}} \le C'' \delta^\gamma$.
By the interpolation inequality \eqref{eq:intpol}, we then obtain $\|\xad-\xdag\|_{\X_r} \le C_5 \delta^{\gamma \frac{u-r}{u+a}}$, and using \eqref{eq:a1} and the triangle inequality yields $\|F(\xad)-\yd\|_\Y \le C_6 \delta$. \qed

\bigskip

Let us now turn to the proof of Theorem~\ref{thm:improve2}.
As a consequence of Lemma~\ref{lem:nonempty}, we can find for any $n \ge 0$ an approximate solution $\xand \in \X(\alpha_n;\yd)$
where $\alpha_n=2^{-n}$.
Using Lemma~\ref{lem:bounds}, one can see that $n^*$ is well-defined by the stopping rule \eqref{eq:stopping} and finite,
once a sequence of approximate solutions $\xand \in \X(\alpha_n;\yd)$ has been chosen.
In addition, we have the following preliminary lower bound.
\begin{lemma} \label{lem:apriori}
Let $n^*$ be chosen by \eqref{eq:stopping}.
Then $\alpha_{n^*} \ge 7\delta^2/M^2$
and $\|\xansd\|_{\X_s}^2 \le 2 M^2$ for all $\xad \in \X(\alpha_{n^*};\yd)$.
As before, $M=\|\xdag\|_{\X_s}$ is the norm of the true solution.
\end{lemma}
\begin{proof}
From Lemma~\ref{lem:bounds}, we deduce that $\|F(\xand)-\yd\|_\Y > 4 \delta$ can only happen, when $\alpha_n > 14 \delta^2/M^2$.
This condition must particularly hold for $n=n^*-1$, which already implies the lower bound for $\alpha_{n^*}$.
The bound for $\|\xansd\|_{\X_s}$ then follows from the second assertion of Lemma~\ref{lem:bounds} with $\tau=7$.
\end{proof}
As a direct consequence of the estimates of Lemma~\ref{lem:rate},
we obtain the following sharper bound for the regularization parameter from below.
\begin{lemma} \label{lem:apriori2}
The regularization parameter determined by \eqref{eq:stopping} satisfies
\begin{align*}
\alpha_{n^*} \ge C_8 \delta^{2-2\gamma\frac{u-s}{a+u}}
\end{align*}
with constant $C_8$ only depending on the norm $\|\xdag\|_{\X_u}$ of the exact solution.
\end{lemma}
As a final ingredient for the proof of Theorem~\ref{thm:improve2},
we require the following estimate for a term that already appeared in the
proof of Lemma~\ref{lem:rate}.
\begin{lemma} \label{lem:rate2}
Assume that $\|F(\xad)-\yd\|_\Y \le 4\delta$ and $\alpha \ge C_8 \delta^{2-2\gamma\frac{u-s}{a+u}}$.
Then
\begin{align*}
2 \|\xdag\|_{\X_u} \|\xad-\xdag\|_{\X_{2s-u}}
\le C_9 \delta^{2\gamma \frac{u-s}{a+u}} + \tfrac{a+2s-u}{2a+2s} \|\xad-\xdag\|^2_{\X_s}
\end{align*}
with constant $C_9$ depending only on the norm $M=\|\xdag\|_{\X_u}$ of the exact solution.
\end{lemma}
\begin{proof}
With similar arguments as in the proof of Lemma~\ref{lem:rate},
we obtain by the interpolation inequality \eqref{eq:intpol} with $p=-a$, $q=2s-u$, $r=s$, the bound
\begin{align*}
&2\|\xdag\|_{\X_u} \|\xad-\xdag\|_{\X_{2s-u}} \\
&\qquad \le 2M\|\xad-\xdag\|^{\frac{u-s}{a+s}}_{\X_{-a}}\|\xad-\xdag\|_{\X_s}^2\big)^{\frac{a+2s-u}{2a+2s}}.
\end{align*}
An application of Young's inequality
$a \cdot b \le \frac{a^p}{p} + \frac{b^q}{q}$
with $p=\frac{2a+2s}{a+u}$ and $q=\frac{2a+2s}{a+2s-u}$ further yields
\begin{align*}
(*)
\le C'(M) \|\xad-\xdag\|_{\X_{-a}}^{2\frac{u-s}{a+u}} + \frac{1}{2} \|\xad-\xdag\|^2_{\X_s}
\end{align*}
with constant $C'(M)$ depending only on $M$. From the stability condition \eqref{eq:a2}
and using $\|\xdag\|_{\X_s} \le \|\xdag\|_{\X_u} = M$,
we deduce that
\begin{align*}
\|\xad-\xdag\|_{\X_{-a}}^{2\frac{u-s}{a+u}}
\le R(M)^{2\frac{u-s}{a+u}} \|F(\xad)-F(\xdag)\|^{2\gamma\frac{u-s}{a+u}}
\le C''(M) \delta^{2\gamma\frac{u-s}{a+u}}.
\end{align*}
Here we employed that $\|F(\xad)-F(\xdag)\|_\Y \le 5 \delta$ which follows from condition \eqref{eq:a1}
and the assumption of the lemma via the triangle inequality.
The assertion of the lemma now follows directly from the previous estimates.
\end{proof}

\noindent
\paragraph{\bf Proof of Theorem~\ref{thm:improve2}.}
From the estimate \eqref{eq:est} which was derived in the proof of Lemma~\ref{lem:rate},
we can deduce that
\begin{align*}
\alpha \|\xad-\xdag\|^2_{\X_s}
\le 2 \delta^2 + 2 \alpha \|\xdag\|_{\X_u} \|\xad-\xdag\|_{\X_{2s-u}}.
\end{align*}
Dividing by $\alpha$ and using Lemma~\ref{lem:rate2} to estimate the last term
yields
\begin{align*}
\tfrac{a+u}{2a+2s} \|\xad-\xdag\|^2_{\X_s} \le 2 \delta^2/\alpha + C_9 \delta^{2\gamma \frac{u-s}{a+u}}.
\end{align*}
From the condition $-a \le s \le u \le 2a+2s$ one can see that $\frac{u+a}{2a+2s} \ge \frac{1}{2}$.
Using this fact, the bound $\alpha=\alpha_{n^*} \ge C_8 \delta^{2-2\gamma \frac{u-s}{a+u}}$ provided by Lemma~\ref{lem:apriori2}, and taking the square root then yields the estimate $\|\xad-\xdag\|_{\X_s} \le C_5 \delta^{\gamma \frac{u-s}{u+a}}$. The estimate $\|\xad-\xdag\|_{\X_{-a}} \le C_5 \delta^{\gamma}$ follows from assumptions \eqref{eq:a2}, \eqref{eq:a2}, and the discrepancy rule, which yield $\|F(\xad)-F(\xdag)\|_\Y \le \|F(\xad)-\yd\|_\Y+\delta \le 5 \delta$. The estimate for $\|\xad-\xdag\|_{\X_r}$ is the again obtained by interpolation.
\qed

\section{Remarks on regularization in Hilbert scales} \label{sec:relation}

Now we are going to recall some previous results about regularization in Hilbert scales and illustrate their relation
to those established in this paper. Let us start with linear inverse problems $T x = y$ with bounded linear forward operators $T$ mapping from $\X$ to $\Y$, and consider data $\yd \in \Y$ satisfying assumption \eqref{eq:a1}.
For the stable solution, we again consider Tikhonov regularization in Hilbert scales for $s \ge 0$ with regularized approximations $\xad$ defined by
\begin{align*}
\xad = \text{argmin} \, \{\|Tx-\yd\|_\Y^2 + \alpha \|x\|^2_{\X_s}\}.
\end{align*}
The basic assumption for the convergence analysis in the linear case is that the operator $T$ admits a two-sided estimate of the form
\begin{align}  \label{eq:twosided}
\underline c \|x\|_{\X_{-a}} \le \|Tx\|_\Y \le \overline c \|x\|_{\X_{-a}},
\end{align}
for some parameter $a \ge 0$ characterizing the degree of ill-posedness of the inverse problem.
If the true solution satisfies $\xdag \in \X_u$ for some $u$ satisfying  $-a \le u \le 2s+a$
and if the regularization parameter is chosen as
\begin{equation} \label{eq:bernd1}
\alpha = \delta^{\frac{2(a+s)}{a+u}}\,,
\end{equation}
then one obtains the optimal convergence rates
\begin{align} \label{eq:bernd2}
\|\xdag-\xad\|_{\X_r} \le C \delta^{\frac{u-r}{a+u}} \qquad \mbox{for all} \qquad -a \le r \le u \le a+2s;
\end{align}
see \cite{Natterer84} or \cite[Rem.~8.24]{EnglHankeNeubauer96} for details.
For $s \le u$ these estimates and also the parameter choice coincide with the ones of Theorem~\ref{thm:improve}
in the case $\gamma=1$, where \eqref{eq:a2} describes the Lipschitz stability of the inverse problem.
Vice versa, assumption \eqref{eq:twosided} also implies the
validity of the condition \eqref{eq:a2} with exponent $\gamma=1$.

\begin{remark} \label{rem:bernd}
The convergence rate result \eqref{eq:bernd2} remains true also in the \emph{oversmoothing} case $s>u \ge -a$,
where $\|\xdag\|_{\X_s}=\infty$; see \cite{HofMat18} for an extension to nonlinear problems.
Also note that in the oversmoothing case with $\alpha$ chosen by \eqref{eq:bernd1}, we have
\begin{align*}
\lim \limits_{\delta \to 0}\frac{\delta^2}{\alpha} =\delta^{\frac{2(u-s)}{a+u}}=+\infty,
\end{align*}
whereas for $u=s,\;\|\xdag\|_{\X_s}< \infty$ and $0<\gamma \le 1$, which is considered in Theorems~\ref{thm:previous} and \ref{thm:rate} as well as e.g.~\cite{ChengYamamoto00},
one obtains
\begin{align*}
\underline c \le  \frac{\delta^2}{\alpha} \le \overline c \quad \mbox{for constants} \quad  0<\underline c \le \overline c < \infty.
\end{align*}
The results of Theorem~\ref{thm:improve} cover the third case $u>s,\;\|\xdag\|_{\X_s}< \infty$, and $0<\gamma \le 1$,
and yield the asymptotic behavior
\begin{align*}
\lim \limits_{\delta \to 0}\frac{\delta^2}{\alpha} =\delta^{\frac{2\gamma(u-s)}{a+u}}=0.
\end{align*}
This is the condition for an a-priori parameter choice that is is required to ensure convergence of Tikhonov regularization in Hilbert spaces; see \cite{EnglHankeNeubauer96}.
\end{remark}

\medskip

Using the interpolation inequality \eqref{eq:intpol} and the first estimate in \eqref{eq:twosided},
we can  also obtain a conditional H\"older stability estimate of the form
\begin{align*}
\|x_1-x_2\|_{\X_{-q}}
\le \|x_1-x_2\|_{\X_{-a}}^{\frac{s+q}{s+a}} \|x_1-x_2\|_{\X_s}^{\frac{a-q}{s+a}} %\qquad -a \le q \le s \\
\le R(\rho) \|T x_1 - T x_2\|_\Y^\gamma,
\end{align*}
which is valid for all $-a \le -q \le s$, for all $x_1,x_2 \in \X_s$ with $\|x_i\|_{\X_s} \le \rho$
and for constants $R(\rho) = 2 \rho^{1-\gamma} \underline c^{-\gamma}$ and $\gamma=\frac{s+q}{s+a}$.
Note that this stability estimate is strictly weaker than \eqref{eq:twosided} whenever $a>q$.
For the case $u=s$, Theorem~\ref{thm:improve} however still yields the same parameter choice
and the order optimal convergence rates as described above.

If $u > s$, then the convergence rates predicted Theorem~\ref{thm:improve} are somewhat smaller;
see also our discussion at the end of Section~\ref{sec:num_smooth}.
In view of Figure~\ref{fig:rates}, the rates are however still order optimal under the given assumptions.
Theorems~\ref{thm:improve2} and \ref{thm:improve} can therefore be understood as true
generalizations, with respect to varying $0<\gamma \le 1$, of the results about Tikhonov regularization in Hilbert scales from \cite{Natterer84} and \cite{EnglHankeNeubauer96}.

\medskip

Let us now turn to operator equations \eqref{eq:ip} with nonlinear operator $F$.
The essential condition for the analysis of Tikhonov regularization in Hilbert scales
used in \cite{Neubauer92} here reads as
\begin{align*} % \label{eq:twoside2}
\underline c \|h\|_{\X_{-a}} \le \|F'(x) h\|_\Y \le \overline c \|h\|_{\X_{-a}} \quad \text{for all }\;\; x \in \D(F) \cap \X_s.
\end{align*}
By Taylor expansion and the triangle inequality, one can see that
\begin{align}
\underline c \|x_1&-x_2\|_{\X_{-a}} \label{eq:nonlin}
\le  \|F(x_1)-F(x_2)\|_\Y \\
&+ \|\int_0^1 [F'(x_1)-F'(x_1+t (x_1-x_2))] (x_1-x_2) dt\|_\Y. \notag
\end{align}
To obtain quantitative estimates for regularization in Hilbert scales,
one should additionally assume that the derivative is at least H\"older continuous,
i.e.,
\begin{align*}
\|[F'(x)-F'(\tilde x)]\, h\|_\Y \le C \|x-\tilde x\|_{\X_r}^\beta \|h\|_{\X_{-a}},
\end{align*}
for all $x,\tilde x \in \D(F) \cap \X_s$ and $h \in \X_{-a}$ with some $-a \le r \le s$ and some $\beta > 0$.
We refer to \cite{Egger07,Neubauer00} for details.
The last term in \eqref{eq:nonlin} can then be absorbed into the left hand side of the estimate.
This yields the stability assumption \eqref{eq:a2} with $\gamma=1$ for all $x_1,x_2 \in \X_s \cap \D(F)$
with $\|x_1\|_{\X_s},\|x_2\|_{\X_s} \le \rho$ and $\|x_1-x_2\|_{-a}$ sufficiently small.
The resulting estimates of Theorem~\ref{thm:improve2} and \ref{thm:improve} then again coincide
with the ones obtained in \cite{Neubauer92} under the nonlinearity conditions stated above.

\bigskip

Let us finally note that the stability condition \eqref{eq:a2} does in general not even imply continuity
or differentiability of the forward operator $F$. Both Theorems~\ref{thm:improve2} and \ref{thm:improve} therefore
cover rather general problems and the results of this paper can be seen as a true generalization of previous
results about regularization in Hilbert scales.

\section{Model problems} \label{sec:problems}

For illustration of the applicability of our results and as model scenarios for our numerical tests,
we consider the following two simple test problems.

\subsection{Data smoothing} \label{sec:datasmoothing}

Let $H^r_{per}(0,2\pi)=\{f(x)=\sum_{k \in \ZZ} f_k e^{i k x} : \|f\|_{r} < \infty\}$ denote the space of $2\pi$-periodic functions with bounded norm $\|f\|_r^2 = \sum_{k} (1+k^2)^r |f_k|^2$.
We consider the reconstruction of a signal $f$ from noisy measurements $f^\delta$ with noise bounded by $\|f^\delta-f\|_{-1} \le \delta$;
this allows for irregular and possibly large noise \cite{Egger08}.
The corresponding inverse problem then reads as follows:
Find $f \in H^0_{per}(0,1)$ such that the corresponding operator $T$ in the linear operator equation is defined by
\begin{align*}
T : H^0_{per}(0,2\pi) \to
% L^2(0,2\pi)=
H^{-1}_{per}(0,2\pi)
, \quad f \mapsto f.
\end{align*}
Note that the  reconstruction will be based on noisy observation $f^\delta$ for the
true solution $f^\dag$. Using interpolation in the frequency domain, one can see that
\begin{align*}
\|f\|_{0} \le \|f\|_{r}^{\frac{1}{r+1}} \|f\|_{-1}^{\frac{r}{r+1}},
\end{align*}
for all $r \ge 0$ and all $f \in H^r_{per}(0,2\pi)$.
By the linearity of the problem and this interpolation estimate,
we obtain for $r=1$ and for all  $f_1,f_2 \in H_{per}^2(0,2\pi)$ that
\begin{align*}
\|f_1-f_2\|_0 \le (\|f_1\|_1+\|f_2\|_1)^{1/2} \|T f_1 - T f_2\|_{-1}^{1/2}.
\end{align*}
This exactly amounts to the stability condition \eqref{eq:a2} with spaces
$\X_s=H^s_{per}(0,2\pi)$, $\Y=H^{-1}_{per}(0,2\pi)$, and parameters $a=0$, $s=1$, $\gamma=1/2$ and $R(\rho)=(2\rho)^{1/2}$.

\begin{remark}
For the data smoothing problem, we also have the two-sided estimate
\begin{align*}
\|f\|_{-1} \le \|Tf\|_\Y \le \|f\|_{-1},
\end{align*}
which amounts to \eqref{eq:twosided} with $a=1$ and which also implies \eqref{eq:a2} with this value $a$
 and $\gamma=1$.
For evaluation of our results,
we will make use of this Lipschitz stability estimate as well as of the H\"older stability estimate above
in our numerical tests.
\end{remark}

\subsection{Parameter identification} \label{sec:parameterid}

As a second model problem, we consider a coefficient inverse problem for a parabolic equation
similar to an example in \cite{HofmannYamamoto10}.
We use different boundary conditions here, which allows us to provide simpler proofs.

Let $\Omega \subset \RR^{d}$ denote a bounded sufficiently regular domain.
We consider a reaction-diffusion problem of the form
\begin{align*}
\partial_t u - \Delta u + c u &= 0 \qquad \text{in } \Omega,\\
\partial_n u &= 0 \qquad \text{on } \partial\Omega,
\intertext{for all $0 < t \le T$ and with given initial values given by}
u(\cdot,0) &= u_0 \quad \; \; \text{in } \Omega.
\end{align*}
We assume that $u_0$ is sufficiently smooth
and that the parameter $c=c(t)$ is independent of space.
The goal here is to identify this parameter $c$ from measurements of $U=\int_\Omega u(x,\cdot) dx$.
This amounts to an inverse problem $F(c)=U^\delta$ with operator
\begin{align*}
F : \D(F) \subset L^2(0,T) \to L^2(0,T), \qquad c \mapsto \int_\Omega u(x,\cdot) dx
\end{align*}
It is not difficult to see that $F$ is well-defined on $\D(F)=\{ c \in L^2(0,T) : c \ge 0\}$.
Using $U=\int_\Omega u(x,\cdot) dx$ and integrating the differential equation over $\Omega$, we obtain
\begin{align*}
U' + c U &= 0, \ t>0
\qquad \text{with} \qquad
U(0)=\int_\Omega u_0(x) dx.
\end{align*}
The solution is then given by $U(t)=U(0) e^{-\int_0^t c(s) ds}$,
which shows that $|U(t)|$ is monotonically decreasing since we assumed $c \ge 0$.
As a consequence, we obtain
\begin{align*}
e^{-\sqrt{T}\|c\|_{L^2(0,T)}} |U(0)| \le |U(t)| \le |U(0)|.
\end{align*}
In particular, $U(t) \ne 0$ whenever $U(0) \ne 0$, which we assume in the following.
Using the differential equation repeatedly, one can obtain estimates for the derivatives
\begin{align*}
% \|U\|_{H^{s+1}(0,T)} \le C(s,T) |U(0)| \|c\|_{H^s(0,T)}^{s+1} \qquad \text{for all } s \ge 0.
\|U'\|_{L^2(0,T)} &\le |U(0)| \|c\|_{L^2(0,T)},\\
\|U''\|_{L^2(0,T)} &\le |U(0)| \big( C_T \|c\|_{H^1(0,T)} \|c\|_{L^2(0,T)} + \|c\|_{H^1(0,T)}\big).
\end{align*}
In a similar way one can estimate $\|U\|_{H^{s+1}(0,T)}$ by powers $\|c\|_{H^s(0,T)}$ for $s \ge 0$.
Now let $U_1,U_2$ denote two measurements resulting from coefficients $c_1,c_2$ but with the same
initial data $U_1(0)=U_2(0)=U(0)$.
Then the difference $W=U_1-U_2$ satisfies
\begin{align*}
W'+c_1 W &= (c_2-c_1) U_2, \ t>0
\qquad \text{and} \qquad
W(0)=0.
\end{align*}
From this equation and the uniform bounds for the functions $F(c_i)=U_i$ one can see that
$F$ is continuous and weakly closed. Using $U(0) \ne 0$ and the monotonicity of $|U(t)|$, one can further deduce that
\begin{align*}
\|c_1-c_2\|_{L^2(0,T)}
&\le \frac{1}{|U_2(T)|} \big(\|W'\|_{L^2(0,T)} + \|c_1\|_{L^2(0,T)} \|W\|_{L^\infty(0,T)} \big) \\
&\le e^{\sqrt{T} \|c_2\|_{L^2(0,T)}} (1+C_T \|c_1\|_{L^2(0,T)}) \|F(c_1)-F(c_2)\|_{H^1(0,T)}.
\end{align*}
For the last step, we used the embedding of $H^1(0,T)$ into $L^\infty(0,T)$ here.
By interpolation of Sobolev spaces, we further obtain
\begin{align*}
\|F(c_1)-F(c_2)\|_{H^1(0,T)}
\le C_s \|F(c_1)-F(c_2)\|_{L^2(0,T)}^{\frac{s}{s+1}} \|F(c_1)-F(c_2)\|_{H^{s+1}(0,T)}^{\frac{1}{s+1}},
\end{align*}
which holds for any parameter $s \ge 0$.
Using the bounds for $U_i=F(c_i)$ to estimate the term $\|F(c_1)-F(c_2)\|_{H^{s+1}(0,T)}$, we then get
\begin{align*}
\|c_1-c_2\|_{L^2(0,T)}
% \le C(\|c_1\|_{L^2},\|c_2\|_{L^2}) \|F(c_1)-F(c_2)\|_{H^{s+1}(0,T)}^{\frac{1}{s+1}} \|F(c_1)-F(c_2)\|_{L^2(0,T)}^{\frac{s}{s+1}} \\
&\le R(\rho) \|F(c_1)-F(c_2)\|_{L^2(0,T)}^{\frac{s}{s+1}}
\end{align*}
for all $\|c_1\|_{H^s(0,T)},\|c_2\|_{H^s(0,T)} \le \rho$ and with appropriate function $R(\rho)$.
This is exactly the conditional stability estimate \eqref{eq:a2} required for our analysis.
% \end{example}

\begin{remark}
One can see that $F$ is also differentiable with $F'(c) h=W$ defined by
\begin{align*}
W' + c W = -h U, \ t>0, \qquad W(0)=0,
\end{align*}
and with function $U=F(c)$ defined as before.
By the variation-of-constants formula, one obtains $W(t)=\int_0^t e^{-\int_s^t c(r) dr} h(s) U(s) ds$,
which shows that
\begin{align*}
\|W\|_{L^\infty(0,T)} \le \|hU\|_{L^1(0,T)} \le \|h\|_{L^1(0,T)} |U(0)|.
\end{align*}
By embedding of Sobolev spaces, one can then obtain an upper estimate
\begin{align*}
\|F'(c) h\|_{L^2(0,T)} = \|W\|_{L^2(0,T)} \le \overline{c} \|h\|_{L^2(0,T)}.
\end{align*}
From the explicit representation of the function $W=F'(c) h$ one can however see
that an estimate $\underline c \|h\|_{L^2(0,T)} \le \|F'(c) h\|_{L^2(0,T)}$ from below is certainly not valid.
Previous results on regularization in Hilbert scales are therefore not applicable directly.
\end{remark}

\section{Numerical tests} \label{sec:numerics}

We now illustrate the results Theorems~\ref{thm:improve2} and \ref{thm:improve}
by some numerical tests. As model problems, we consider the ones introduced in the previous section.

\subsection{Data smoothing} \label{sec:num_smooth}

As a first test case, we consider the problem described in Section~\ref{sec:datasmoothing}
and we utilize the Lipschitz stability condition
\begin{align} \label{eq:stabex1}
\|f_1-f_2\|_{-1} \le \|T f_1 - T f_2\|_{-1}.
\end{align}
This corresponds to \eqref{eq:a2} with $a=1$ and $\gamma=1$, and we may set
$s \ge 0$ arbitrary.
For our numerical tests, the functions $f$ are represented by piecewise linear splines
over a uniform grid of $[0,2\pi]$.
As reference solutions, we consider the functions
\begin{align*}
f^\dag(t) = \begin{cases} 0, & t<\pi, \\ 1, & t>\pi, \end{cases}
\qquad
f^\dag(t)=\sqrt{t (2\pi-t)},
\quad \text{and} \quad
f^\dag(t) = \begin{cases} t, & t<\pi, \\ 2\pi-t, & t>\pi, \end{cases}
\end{align*}
which have different regularity.
In the first case, $\fdag \in H^{u}_{per}(0,2\pi)$ for any $u < 1/2$,
in the second case, $\fdag \in H^{u}_{per}(0,2\pi)$ for all $u < 1$,
and in the third case, $\fdag \in  H^{u}_{per}(0,2\pi)$ for all $u < 3/2$.
In Table~\ref{tab:1} we list the parameter choices and convergence rates for these
three cases predicted by our theory.

\begin{table}[ht!]
\def\arraystretch{1.1}
\small
 \begin{tabular}{c||c|c|c||c|c|c}
                 & \multicolumn{3}{|c||}{$s=0$}
                 & \multicolumn{3}{|c}{$s=1$} \\
\hline
 $u$             & $\alpha$               & $\|\fad-\fdag\|_{\X_0}$   & $\|\fad-\fdag\|_{\X_1}$
                 & $\alpha$               & $\|\fad-\fdag\|_{\X_0}$   & $\|\fad-\fdag\|_{\X_1}$ \\
 \hline \hline
 $\frac{1}{2}$   & $\delta^{4/3}$         & $\delta^{1/3}$           & {\blue---}
                 & $\delta^{8/3}$         & $\blue \delta^{1/3}$     & {\blue---} \\
 \hline
 $1$             & $\delta$               & $\delta^{1/2}$           & $\blue \delta^0$
                 & $\delta^{2}$           & $\delta^{1/2}$           & $\delta^0$             \\
 \hline
 $\frac{3}{2}$   & $\delta^{4/5}$         & $\blue \delta^{3/5}$     & $\blue \delta^{1/5}$
                 & $\delta^{8/5}$         & $\delta^{3/5}$           & $\delta^{1/5}$
 \end{tabular}
\medskip
\caption{\small Parameter choice and convergence estimates of Theorem~\ref{thm:improve} for the data smoothing problem
with assumption \eqref{eq:a2} with $a=1,\gamma=1$.
The results in blue are not covered by our theory since one of the conditions $s \le u \le 2s+a$ or $-a \le r \le s$ is violated.\label{tab:1}}
% Results for which $-a \le r \le u$ or $s \le u \le 2s+a$ is violated are skipped.\label{tab:1}}
\end{table}

Note that due to the restrictions $-a \le r \le s$ and $s \le u \le 2s+a$, not all results listed in the table are fully covered by our theory.
In Tables~\ref{tab:2}, we list the results obtained in our numerical tests for the a-priori parameter choice.

\begin{table}[ht!]
\def\arraystretch{1.1}
\small
 \begin{tabular}{c||c|c|c||c|c|c}
                 & \multicolumn{3}{|c||}{$s=0$}
                 & \multicolumn{3}{|c}{$s=1$} \\
\hline
 $u$             & $\alpha$                & $\|\fad-\fdag\|_{\X_0}$     & $\|\fad-\fdag\|_{\X_1}$
                 & $\alpha$                & $\|\fad-\fdag\|_{\X_0}$     & $\|\fad-\fdag\|_{\X_1}$ \\
 \hline \hline
 $\frac{1}{2}$   & $\delta^{1.33}$         & $\delta^{0.41}$            & {\blue---}
                 & $\delta^{2.67}$         & $\blue \delta^{0.37}$      & {\blue---}                    \\
 \hline
 $1$             & $\delta^{1.00}$         & $\delta^{0.53}$            & $\blue \delta^{0.04}$
                 & $\delta^{2.00}$         & $\delta^{0.49}$            & $\delta^{0.04}$        \\
 \hline
 $\frac{3}{2}$   & $\delta^{0.80}$         & $\blue \delta^{0.67}$      & $\blue \delta^{0.23}$
                 & $\delta^{1.60}$         & $\delta^{0.64}$            & $\delta^{0.21}$
 \end{tabular}
\medskip
\caption{\small Rates for a-priori parameter choice $\alpha=\delta^{2-2\gamma \frac{u-s}{u+a}}$
and reconstruction errors obtained in numerical tests for data smoothing problem with $a=1$ and $\gamma=1$.
Results in red and blue are not covered by theory.\label{tab:2}}
\end{table}

The observed convergence rates agree very well with the theoretical predictions.
Although the condition $u \le 2s+a$ is violated for the results of the last line for $s=0$,
we still observe the optimal convergence rates also in that case.
The results that are skipped correspond to negative rates or very large errors.
In Table~\ref{tab:3}, we list the corresponding results obtained with the a-posteriori parameter choice.

\begin{table}[ht!]
\def\arraystretch{1.1}
\small
 \begin{tabular}{c||c|c|c||c|c|c}
                 & \multicolumn{3}{|c||}{$s=0$}
                 & \multicolumn{3}{|c}{$s=1$} \\
\hline
 $u$             & $\alpha$              & $\|\fad-\fdag\|_{\X_0}$   & $\|\fad-\fdag\|_{\X_1}$
                 & $\alpha$              & $\|\fad-\fdag\|_{\X_0}$   & $\|\fad-\fdag\|_{\X_1}$ \\
 \hline \hline
 $\frac{1}{2}$   & $\delta^{1.26}$       & $\delta^{0.36}$          & {\blue---}
                 & $\delta^{2.59}$       & $\blue \delta^{0.36}$    & {\blue---}                    \\
 \hline
 $1$             & $\delta^{1.02}$       & $\delta^{0.48}$          & $\blue \delta^{0.02}$
                 & $\delta^{1.95}$       & $\delta^{0.48}$          & $\delta^{0.04}$        \\
 \hline
 $\frac{3}{2}$   & $\delta^{1.00}$       & $\blue \delta^{0.56}$    & $\blue \delta^{0.01}$
                 & $\delta^{1.48}$       & $\delta^{0.58}$          & $\delta^{0.19}$
 \end{tabular}
\medskip
\caption{\small Parameter choice and reconstruction errors for a-posteriori parameter choice
for data smoothing problem using \eqref{eq:a2} with $a=1,\gamma=1$.
The results in red and blue are not covered by our theory since  one of the conditions $s \le u \le 2s+a$ or $-a \le r \le s$ is violated.\label{tab:3}}
\end{table}

Also here we observe the optimal convergence rates in good agreement with the theoretical predictions of Theorem~\ref{thm:improve2}.
For the case $s=0$ and $u=3/2$, the condition $u \le 2s+a$ is violated
and we here observe a saturation phenomenon, i.e., for the choice $s=0$ smoothness higher than $u=1$
does not lead to a further improvement. This was not the case for the a-priori stopping rule.
However, one can take advantage of higher smoothness here by regularizing in a stronger norm, e.g., with $s=1$,
which restores the full convergence rate.

\bigskip

Before we proceed to the second test problem, let us make another observation.
As outlined in the previous section, we can also use the H\"older stability estimate
\begin{align} \label{eq:stabex1a}
\|f_1-f_2\|_{0} \le R(\rho) \|T f_1 - T f_2\|_{-1}^{1/2}
\end{align}
for all $f_1,f_2 \in H^1_{per}(0,2\pi)$ with $\|f_i\|_{1} \le \rho$ instead of the Lipschitz estimate \eqref{eq:stabex1}.
This amounts to the stability condition \eqref{eq:a2} with $a=0$, $\gamma=1/2$, and $s =1$.
The theoretical rates predicted by Theorems~\ref{thm:improve} and \ref{thm:improve2} are depicted in Table~\ref{tab:4}.
\begin{table}[ht!]
\def\arraystretch{1.1}
\small
 \begin{tabular}{c||c|c|c||c|c|c}
                 & \multicolumn{3}{|c||}{$s=0$}
                 & \multicolumn{3}{|c}{$s=1$} \\
\hline
 $u$             & $\alpha$          & $\|\fad-\fdag\|_{\X_0}$ & $\|\fad-\fdag\|_{\X_1}$
                 & $\alpha$          & $\|\fad-\fdag\|_{\X_0}$ & $\|\fad-\fdag\|_{\X_1}$ \\
 \hline \hline
 $\frac{1}{2}$   & $\delta$          & {\blue---}             & {\blue---}
                 & $\delta^3$        & {\blue---}             & {\blue---} \\
 \hline
 $1$             & $\delta$          & $\blue \delta^{1/2}$   & $\blue \delta^0$
                 & $\delta^{2}$      & $\delta^{1/2}$         & $\delta^0$   \\
 \hline
 $\frac{3}{2}$   & $\delta$          & $\blue \delta^{1/2}$   & $\blue \delta^{1/6}$
                 & $\delta$          & $\delta^{1/2}$         & $\delta^{1/6}$
 \end{tabular}
\medskip
\caption{Parameter choice and convergence estimates of Theorem~\ref{thm:improve} for the data smoothing problem
using \eqref{eq:a2} with $a=0,\gamma=\frac{1}{2}$.
Results in blue are not covered by our theory since one of the conditions $s \ge 1$, $s \le u \le 2s+a$, or $-a \le r \le s$ is violated. The first condition is needed here additionally for the stability condition \eqref{eq:stabex1a}.\label{tab:4}}
\end{table}

Due to the choice of the parameters $a$, $\gamma$, and $s$ involved in the conditional stability estimate,
we also obtain a  different range of applicable smoothness indices here.
Also note that the convergence rates for $u=3/2$ and $r=1$
are somewhat smaller than those obtained for $a=1$ and $\gamma=1$, which is explained in Figure~\ref{fig:2}.

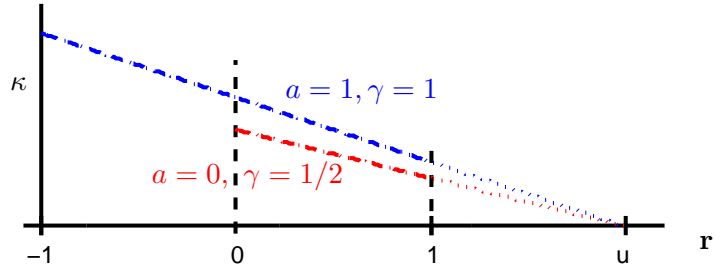
\begin{figure}[ht]
\includegraphics[width=0.6\textwidth]{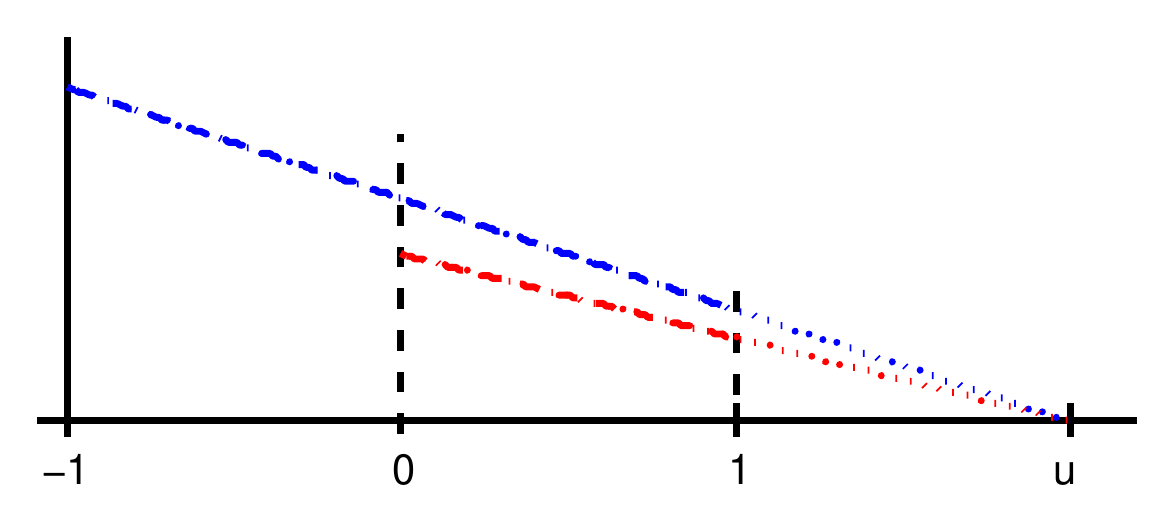}
% \put(-352,45){\color{black} \bf \small $\|f_1-f_2\|_{r} \le R(\rho) \|Tf_1-Tf_2\|_{-1}^{\gamma}$}
% \put(-240,80){\color{black} \bf \small $\|f\|_{r} \le R(\rho) \|Tf\|_{-1}^{\gamma}$}
\put(-253,73){\color{black} \small $\kappa$}
\put(-200,38){\color{red} \bf \small $a=0,\,\,\gamma=1/2$}
\put(-150,70){\color{blue}\bf \small $a=1,\gamma=1$}
\put(5,13){\color{black} \small $\bf r$}
\caption{Exponents $0 \le \kappa \le 1$ in $\|\fad-\fdag\|_{\X_r}=O(\delta^\kappa)$ as $\delta~\to~0$
corresponding to Theorem~\ref{thm:improve} for data smoothing problem
with different parameters $a=1$, $\gamma=1$ (blue) and $a=0$, $\gamma=1/2$ (red, dashed).
% The range $s \le u \le 2s+a$ of admissible smoothness indices is plotted along the x-axis.
% The relation of the H\"older stability condition \eqref{eq:stabex1a}
% and the Lipschitz stability condition \eqref{eq:stabex1} is depicted in black dashed.
\label{fig:2}}
\end{figure}

As can be seen from the plot, the interpolation estimate used to derive the H\"older stability condition \eqref{eq:stabex1a}
from the Lipschitz estimate \eqref{eq:stabex1} is suboptimal when $s<u$. In addition, also the range of admissible smoothness indices $s \le u \le 2s+a$ shrinks when increasing $a$.
%
% A general guideline for the selection of the parameters in the stability condition \eqref{eq:a2}
% therefore is to choose $a$ as small and $\gamma$ as large as possible.
%
A general guideline for the choice of parameters $a$ and $\gamma$ in the stability condition \eqref{eq:a2} is therefore
to choose these parameters as large as possible. The parameter $s$ in the regularization terms should also be chosen
as large as possible, but small enough, such that $\xdag \in \X_u$ can be expected for some $u \ge s$.

\subsection{Parameter identification}

The parameter identification problem discussed in Section~\ref{sec:parameterid} can be phrased as
$F(c)=U^\delta$ with $F(c)=U$ where $U$ solves
\begin{align} \label{eq:U}
U'(t) + c(t) U(t) = 0, \qquad U(0) \ne 0 \quad \text{given}.
\end{align}
We set $X_0=L^2(0,T)$ with norm $\|c\|_{0}^2 = \int_0^T c(s)^2 ds$ and
and $X_2=H^2(0,T)$ with norm $\|c\|_{2}^2 = \|c\|^2_0 +  \|c''\|_{0}^2$.
For any parameter $0 \le s \le 2$, we then obtain by interpolation that
$X_s=H^s(0,T)$. As shown in the previous section, we have
\begin{align*}
\|c_1-c_2\|_{L^2(0,T)}
&\le R(\rho) \|F(c_1)-F(c_2)\|_{L^2(0,T)}^{\frac{s}{s+1}}
\end{align*}
here for all $c_1,c_2 \in H^s(0,T)$ with $c_i \ge 0$ and $\|c_1\|_{H^s(0,T)},\|c_2\|_{H^s(0,T)} \le \rho$ and with appropriate function $R(\rho)$.
In our tests, we will set $s=1$ or $s=2$.

The evaluation of $F$ requires the solution of the initial value problem \eqref{eq:U}.
For this we approximate $U$ by continuous piecewise linear functions over a regular grid of $[0,T]$,
and we represent $c$ by cubic splines over the same grid.
The problem \eqref{eq:U} is then solved by a Petrov-Galerkin method using discontinuous piecewise constant test functions.
This allows us to exactly evaluate the derivative $W=F'(c) h$ and the adjoint $d=F'(c)^* r$ on the discrete level.
For the computation of approximate minimizers for the Tikhonov functional,
we then utilize a Gau\ss-Newton method.

To evaluate the convergence behavior of our regularization strategy,
we again consider three reference solutions of different smoothness, defined by
\begin{align*}
c^\dag(t) = \begin{cases} t, & t<T/2, \\ T-t, & t>T/2, \end{cases}
\qquad
c^\dag(t)=t \sqrt{t},
\qquad \text{and} \qquad
c^\dag(t) = t(T-t).
\end{align*}
Here $c^\dag \in \X_u$ for all $u<3/2$ in the first case
and $c^\dag \in \X_u$ for all $u<2$ in the second.
Note that we only have $c^\dag \in \X_u$ for all $u<5/2$ in the third case, since $X_s \neq H^s(0,T)$ for $s > 5/2$
due to the appearance of additional boundary conditions; see \cite{Neubauer88} for details.
The convergence rates of Theorem~\ref{thm:improve} are listed in Table~\ref{tab:5}.

\begin{table}[ht!]
\def\arraystretch{1.1}
\small
 \begin{tabular}{c||c|c|c||c|c|c}
               & \multicolumn{3}{|c||}{$s=1$}
               & \multicolumn{3}{|c}{$s=2$} \\
\hline
 $u$           & $\alpha$             & $\|\cad-\cdag\|_{\X_0}$ & $\|\cad-\cdag\|_{\X_1}$
               & $\alpha$             & $\|\cad-\cdag\|_{\X_0}$ & $\|\cad-\cdag\|_{\X_1}$ \\
 \hline \hline
 $\frac{3}{2}$ & $\delta^{5/3}$       & $\delta^{1/2}$         & $\delta^{1/6}$
               & $\delta^{22/9}$      & {\blue---}             & {\blue---} \\
 \hline
 $2$           & $\delta^{3/2}$       & $\delta^{1/2}$         & $\delta^{1/4}$
               & $\delta^{2}$         & $\delta^{2/3}$         & $\delta^{1/3}$  \\
 \hline
 $\frac{5}{2}$ & $\blue \delta^{7/5}$ & $\blue \delta^{1/2}$   & $\blue \delta^{3/10}$
               & $\delta^{26/15}$     & $\delta^{2/3}$         & $\delta^{2/5}$
 \end{tabular}
\medskip
\caption{Parameter choice and convergence estimates of Theorem~\ref{thm:improve} for the parameter identification problem using
the stability condition \eqref{eq:a2} with $a=0$ and $\gamma=s/(s+1)$.
Results in blue are not covered since $s \le u \le 2s+a$ or $-a \le r \le s$ is violated.\label{tab:5}}
% Results for which $-a \le r \le u$ or $s \le u \le 2s+a$ is violated are skipped.\label{tab:1}}
\end{table}

In Table~\ref{tab:6} we list the reconstruction errors obtained in our numerical tests with a-posteriori parameter choice strategy according to Theorem~\ref{thm:improve2}.

\begin{table}[ht!]
\def\arraystretch{1.1}
\small
 \begin{tabular}{c||c|c|c||c|c|c}
                 & \multicolumn{3}{|c||}{$s=1$}
                 & \multicolumn{3}{|c}{$s=2$} \\
\hline
 $u$             & $\alpha$              & $\|\fad-\fdag\|_{\X_0}$   & $\|\fad-\fdag\|_{\X_1}$
                 & $\alpha$              & $\|\fad-\fdag\|_{\X_0}$   & $\|\fad-\fdag\|_{\X_1}$ \\
 \hline \hline
 $\frac{3}{2}$   & $\delta^{1.44}$       & $\delta^{0.52}$          & $\delta^{0.17}$
                 & $\delta^{2.34}$       & $\blue \delta^{0.64}$    & $\blue \delta^{0.22}$  \\
 \hline
 $2$             & $\delta^{1.46}$       & $\delta^{0.54}$          & $\delta^{0.18}$
                 & $\delta^{2.04}$       & $\delta^{0.72}$          & $\delta^{0.37}$        \\
 \hline
 $\frac{5}{2}$   & $\blue \delta^{1.56}$ & $\blue \delta^{0.57}$    & $\blue \delta^{0.19}$
                 & $\delta^{1.65}$       & $\delta^{0.69}$          & $\delta^{0.41}$
 \end{tabular}
\medskip
\caption{\small Parameter choice and reconstruction errors for data smoothing problem using \eqref{eq:a2} with $a=1$ and $\gamma=1$ and a-posteriori parameter choice.
Results in blue are not covered by theory since one of the conditions $s \le u \le 2s+a$ or $-a \le r \le s$ is violated.\label{tab:6}}
\end{table}

Again, the results are in very good agreement with the theoretical predictions displayed in Table~\ref{tab:5}.
Similar results were also obtained for the a-priori parameter choice but they are omitted here.

\section{Discussion}

In this paper, we investigated Tikhonov regularization in Hilbert scales under a conditional stability assumption
for the considered inverse problem. Optimal convergence rates were established for a-priori
and a-posteriori parameter choice strategies. Apart from the conditional stability estimate, no
further assumptions on the continuity or differentiability of the operator $F$ were required.

For the statement of our main results, we utilized here the framework of Hilbert scales.
This allowed us to keep the presentation compact and to discuss in detail the relation
to previous work. An extension of the analysis to Banach scales seems to be possible
and will be a topic of future research.

Conditional stability estimates have been used recently for the convergence analysis of Landweber iteration \cite{deHoopQiuScherzer12}.
It has been observed there that a H\"older stability condition \eqref{eq:a2} with $\gamma>1/2$ together with the usual continuity assumptions
already implies the tangential cone condition.
Stability conditions for linear or linearized problems have also been used for the convergence analysis of iterative
regularization methods in Hilbert scales; see e.g. \cite{Egger07,Neubauer00}.
An extension of such a convergence analysis under weaker H\"older stability conditions, which have been utilized in this paper, may be possible and should also be addressed in the future.

\section*{Acknowledgments}
This work was initiated at the Workshop ``Inverse Problems in the Alps'' organized by O. Scherzer and R. Ramlau from RICAM, Vienna/Linz and supported by the Austrian Science Foundation (FWF). Both authors further gratefully acknowledge support by the German Research
Foundation (DFG) via grants IRTG~1529 and TRR~154 (Herbert Egger) as well as HO 1454/10-1 and 12-1 (Bernd Hofmann).
Moreover the first author was supported by the ``Excellence Initiative'' of the German Federal and State Governments via the Graduate School of Computational Engineering GSC~233 at TU Darmstadt.

% \nocite{*}
% \bibliographystyle{abbrv

\end{document}